\documentclass[leqno,11pt]{article}
\usepackage{amssymb,amsmath,amsfonts,amsthm}
\usepackage{authblk,  empheq}
\usepackage{epsfig,color,mathrsfs}
\usepackage{graphicx}
\usepackage{amscd}
\usepackage{caption}
\usepackage{mathrsfs,bbm}
\usepackage{multicol}
\usepackage{color}

\newcommand{\be}{\begin{equation}}
\newcommand{\ee}{\end{equation}}
\newcommand{\ben}{\begin{eqnarray*}}
\newcommand{\een}{\end{eqnarray*}}

\newcommand{\ds}{\displaystyle}

\newcommand{\intl}{\int\limits}

\newcommand{\R}{\mathbb R}

\allowdisplaybreaks

\newtheorem{theorem}{Theorem}[section]

\newtheorem{corollary}[theorem]{Corollary}

\newtheorem{definition}[theorem]{Definition}

\newtheorem{lemma}[theorem]{Lemma}

\newtheorem{proposition}[theorem]{Proposition}
\newtheorem{remark}[theorem]{Remark}

\numberwithin{equation}{section}

\begin{document}
\title{Superharmonic functions in the upper half space with a nonlocal boundary condition} 

\author[ ]{Marius Ghergu$^{1,2}$\,\footnote{ORCID: 0000-0001-9104-5295} }

\affil[ ]{$^1$School of Mathematics and Statistics}
\affil[ ]{University College Dublin, Belfield Campus, Dublin 4, Ireland} 
\affil[ ]{}

\affil[ ]{$^2$Institute of Mathematics of the Romanian Academy}
\affil[ ]{21 Calea Grivitei St.,  010702 Bucharest, Romania}
\affil[ ]{E-mail: {\tt marius.ghergu@ucd.ie}}



\maketitle

\begin{abstract} 
We discuss the existence of positive superharmonic functions $u$ in $\R^N_+=\R^{N-1}\times (0, \infty)$, $N\geq 3$, in the sense $-\Delta u=\mu$ for some Radon measure $\mu$, so that $u$ satisfies the nonlocal boundary condition 
$$
\frac{\partial u}{\partial n}(x',0)=\lambda \intl_{\R^{N-1}}\frac{u(y',0)^p}{|x'-y'|^k}dy' \quad\mbox{ on }\partial \R^N_+,
$$
where $p,\lambda>0$ and $k\in (0, N-1)$. First, we show that no solutions exist if $0<k\leq 1$. Next, if $1<k<N-1$, we obtain a new critical exponent  given by $p^*=\frac{N-1}{k-1}$ for the existence of such solutions. If $\mu\equiv 0$ we construct an exact solution for $p>p^*$ and discuss the existence of regular solutions, case in which we identify a second critical exponent given by $p^{**}=2\cdot \frac{N-1}{k-1}-1$. Our approach combines various integral estimates with the properties of the newly introduced $\alpha$-lifting operator and fixed point theorems. 
\end{abstract}

\noindent{\bf Keywords: } Superharmonic functions in the half space; measure data; $\alpha$-lifting operator; critical exponent; regular solutions.
\medskip

\noindent{\bf 2020 AMS MSC: } 45M20;  45G05; 31B10; 45E10; 45M05


\section{Introduction and the main results}
Let $N\geq 3$ be an integer and denote by $\R^N_+=\R^{N-1}\times (0, \infty)$ the upper half space in $\R^N$ whose boundary is $\partial \R^{N}_+=\R^{N-1}\times \{0\}\simeq \R^{N-1}$. The points  $x\in \R^N$ will be denoted by $x=(x', x_N)$, where $x'\in \R^{N-1}$ and $x_N\in \R$.  In this paper we investigate the existence of positive solutions to the problem
\begin{equation}\label{half1}
\begin{cases}
\ds -\Delta u=\mu &\quad\mbox{ in }\R^N_+,\\[0.2cm]
\ds \frac{\partial u}{\partial n}(x',0)=\lambda \intl_{\R^{N-1}}\frac{u(y',0)^p}{|x'-y'|^k}dy'&\quad\mbox{ on }\partial \R^N_+,
\end{cases}
\end{equation}
where $p, \lambda>0$, $k\in (0, N-1)$ and $\mu$ is a Radon measure on $\R^N$ whose support is contained in the open half space $\R^N_+$. We also assume that the Newtonian potential of $\mu$ given by
\begin{equation}\label{um}
U^\mu(x)=\frac{1}{(N-2)\sigma_N}\intl_{\R^N}\frac{d\mu(y)}{|x-y|^{N-2}},
\end{equation}
is finite for all $x\in \R^N$.
In the above, $\sigma_N$ denotes the surface area of the $N-1$ dimensional unit sphere while in \eqref{half1}, $n=(0,0,\dots, 0, -1)\in \R^N$ denotes the outer unit normal at $\partial \R^{N}_+$, so that 
$$
\ds \frac{\partial u}{\partial n}(x',0)=-\frac{\partial u}{\partial x_N}(x',0).
$$
The study of \eqref{half1} is motivated by a number of works started in the mid 1990s related to harmonic functions in the upper half space which satisfy some nonlinear boundary conditions.  Using the moving plane method, Hu \cite{H94} obtained that the problem 
\begin{equation}\label{half2}
\begin{cases}
\ds \Delta u=0 &\quad\mbox{ in }\R^N_+,\\[0.2cm]
\ds \frac{\partial u}{\partial n}=\lambda u^p &\quad\mbox{ on }\partial \R^N_+,
\end{cases}
\end{equation}
has no positive solutions if $1\leq p<\frac{N}{N-2}$ and $\lambda>0$. Ou \cite{O96} proved that the exponent $\frac{N}{N-2}$ is optimal as for $p=\frac{N}{N-2}$ and $\lambda>0$, all positive solutions of \eqref{half2} are of the form $c|x-x_0|^{2-N}$ for some $x_0\in \R^N\setminus\overline{\R^N_+}$ and $c>0$. A full characterization of solutions to \eqref{half2} in the case $p=\frac{N}{N-2}$ and $\lambda\in \R$ is obtained in Li and Zhu \cite{LZ95} using the method of moving spheres. Precisely, it is obtained in \cite{LZ95} that:
\begin{itemize}
\item If $\lambda\leq 0$ then all nonnegative solutions of \eqref{half2} are of the form $u(x)=ax_N+b$ where $a,b\geq 0$, $a=-\lambda b^{\frac{N}{N-2}}$.

\item If $\lambda>0$ then all nonnegative solutions of \eqref{half2} are of the form
$$
u(x)=\Big(\frac{\varepsilon}{(x_N+\frac{\varepsilon\lambda}{N-2})^2+|x'-x_0'|^2}\Big)^{\frac{N-2}{2}}\quad \mbox{ where }\varepsilon>0.
$$
\end{itemize}
The case $\lambda=1$ and $p>\frac{N}{N-2}$ was discussed in \cite{H14}. 
Further, the more general problem 
$$
\begin{cases}
\ds \Delta u=au^q &\quad\mbox{ in }\R^N_+,\\[0.2cm]
\ds \frac{\partial u}{\partial n}=\lambda u^p &\quad\mbox{ on }\partial \R^N_+,
\end{cases}
$$
was investigated in \cite{CSF96, LZ03, LZ99, TWZ22}. The existence of sign-changing solutions to
$$
\begin{cases}
\ds \Delta u=0 &\quad\mbox{ in }\R^N_+,\\[0.2cm]
\ds \frac{\partial u}{\partial n}+\lambda u=u |u|^{p-1}+f(x) &\quad\mbox{ on }\partial \R^N_+,
\end{cases}
$$
was obtained in \cite{FMM13} by using the contraction mapping argument under the hypothesis of small data $f\in L^{\frac{(p-1)(N-1)}{p}}(\R^{N-1})$ and $p>\frac{N-1}{N-2}$. Different results, based on Fourier analysis method,  were obtained in \cite{CF21}.

In the present work, we investigate superharmonic functions in the half space $\R^N_+$ which satisfy a new type of nonlinear condition on $\partial \R^N_+$. Unlike the boundary condition imposed in \eqref{half2}, we prescribe a nonlocal condition given by
$$
\frac{\partial u}{\partial n}(x',0)=H_u(x') \quad\mbox{ on }\partial \R^N_+,
$$
where $H_u$ is the convolution  term
\begin{equation}\label{hu}
H_u(x') :=u(\cdot, 0)^p\ast |x'|^{-k}=\intl_{\R^{N-1}}\frac{u(y',0)^p}{|x'-y'|^k}dy',\quad x'\in \R^{N-1}.
\end{equation}
We  emphasize a new critical exponents which separates the existence and nonexistence regions of positive solutions to \eqref{half1}. Our approach is new and differs from the previous methods employed for instance in \cite{CF21, CL24, CL21, CLO06, CLO07, FMM13, LL16, YYL04, LZ03, LL20, TWZ22} as follows:
\begin{itemize}
\item We introduce the $\alpha$-lifting operator $J_\alpha:L^q(\R^{N-1})\to C^{0,\gamma}(\R^N)$ for suitable $q>1>\gamma>0$, see \eqref{lift1} in Section 2.2. This allows us to formulate the positive solutions of \eqref{half1} in terms of integral operators.
\item The existence of positive solutions to \eqref{half1} is achieved in the case $\mu\not\equiv 0$ by means of the Schauder Fixed Point Theorem in a subspace of $C(\overline B_R)$, for some $R>0$.
\end{itemize}
Recall that 
$$
\Phi(x)=\frac{1}{(N-2)\sigma_N}|x|^{2-N},
$$ 
is the fundamental solution of the Laplace operator in $\R^N$ in the sense that
$$
-\Delta \Phi=\delta_0\quad\mbox{ in }\mathcal{D}'(\R^N),
$$
where $\delta_0$ is the Dirac delta measure in $\R^N$ concentrated at the origin. Let
  
\begin{equation}\label{green}
G(x,y)=\Phi(x-y)+\Phi(\overline x-y), \quad x,y\in \overline{\R^N_+},
\end{equation} 
be the Green function of $-\Delta$ in $\R^N_+$, subject to Neumann boundary condition at $\partial\R^N_+$. In \eqref{green}, 
$\overline  x=(x', -x_N)$ denotes the  reflection of $x=(x', x_N)\in \overline{\R^N_+}$ with respect to the $x_N=0$ hyperplane.  Let us just recall that $G$ satisfies
$$
\begin{cases}
-\Delta G(\cdot, y)=\delta_y\quad\mbox{ in }\mathcal{D}'(\R^N_+)\\[0.2cm]
\ds \frac{\partial G}{\partial n}(\cdot,y) =0 \quad\mbox{ on }\partial \R^N_+ 
\end{cases}
\quad\mbox{ for all }y\in \R^N_+.
$$
Using the Green identities, we can formally convert \eqref{half1} into the integral equation 
\begin{equation}\label{l20}
u(x)=\intl_{\R^N_+}G(x,y)d\mu(y)+\lambda \intl_{\R^{N-1}} G(x, (y',0)) H_u(y') dy' \quad\mbox{ a.e. in }\R^N_+,
\end{equation}
that is,
\begin{equation}\label{l2}
u(x)=\intl_{\R^N_+} G(x,y)d\mu+\frac{2\lambda}{(N-2)\sigma_N} \intl_{\R^{N-1}} \intl_{\R^{N-1}} 
\frac{u(z',0)^p}{|x-(y',0)|^{N-2}|y'-z'|^k}dz' dy',
\end{equation}
almost everywhere in $\R^N_+$. 
The precise meaning of the above equality is given in the definition below.

\begin{definition} We say that a measurable function $u$ is a positive solution of \eqref{half1} if:
\begin{enumerate}
\item[\rm (i)] $u>0$ a.e. in $\R^N_+$ and for almost all $x'\in \R^{N-1}$ we have 
\begin{equation}\label{l1}
\lim_{z\in \R^N_+, \, z\to (x',0) } u(z)=u(x',0)<\infty.
\end{equation}
\item[\rm (ii)] $H_u$ defined in \eqref{hu} is finite a.e. in $\R^{N-1}$ and 
\begin{equation}\label{l10}
\intl_{\R^{N-1}} \frac{H_u(y') dy'}{|x-(y',0)|^{N-2}}<\infty \quad \mbox{\; for almost all  }x\in \R^N_+. 
\end{equation}
\item[\rm (iii)] $u$ satisfies \eqref{l2} a.e. in $\R^N_+$. 
\end{enumerate}
\end{definition}
The above definition of a solution to \eqref{half1} does not require extra regularity on $u$, such as $u\in W^{1,q}_{loc}(\R^N_+)$.
Precisely, condition \eqref{l1} states that $u(x',0)$ is well defined on $\partial \R^N_+$ as the limit of $u(z)$ when $z\in \R^N_+$  approaches $(x', 0)$. We do not need to resort to the trace operator which assumes $u$ lie in a certain locally Sobolev space. 
Finally, \eqref{l10} is a natural  requirement to make the last integral in \eqref{l2} finite. We shall see that our solutions to \eqref{half1} are constructed such that $u-\int_{\R^N_+} G(\cdot, y) d\mu(y) $ is  H\"older continuous  in $\overline{\R^N_+}$; see \eqref{vgre} below.
Our first result discusses the existence of a positive solution in the case $\mu\not\equiv 0$.
\begin{theorem}\label{thm1} {\sf (Case $\mu\not\equiv 0$)}
\smallskip

Assume $\mu\not\equiv 0$, $k\in (0, N-1)$ and $p>0$.
\begin{enumerate}
\item[\rm (i)] If one of the following conditions hold:
\begin{itemize}
\item[\rm (i1)] $0<k\leq 1$ and $p>0$;
\item[\rm (i2)] $1<k<N-1$ and $0<p<\frac{N-1}{k-1}$;
\end{itemize}
then, for all $\lambda>0$ the problem \eqref{half1} has no positive solutions.
\item[\rm (ii)] If $1<k<N-1$ and $p=\frac{N-1}{k-1}$ then,  for all $\lambda>0$ the problem \eqref{half1} has no positive solutions $u$ with $u(\cdot, 0)\in L^p(\R^{N-1})$.
\item[\rm (iii)] If $1<k<N-1$, $p>\frac{N-1}{k-1}$  and $\mu$ satisfies 
\begin{equation}\label{munu}
U^{\mu}(x)\leq C(N,\mu)(1+|x|)^{1-k}\quad\mbox{ a.e. in } \R^N,
\end{equation}
for some constant $C(N, \mu)>0$, then,  there exists $\lambda^*>0$ such that the problem \eqref{half1} has a positive solution $u$ for all $0<\lambda<\lambda^*$. 
Moreover, 
\begin{equation}\label{vgre}
u-\intl_{\R^N_+} G(\cdot, y) d\mu(y) \in C^{0,\gamma}(\overline{\R^N_+})\quad\mbox{ for some }\gamma\in (0,1).
\end{equation}
\end{enumerate} 
\end{theorem}
Theorem \ref{thm1} highlights the role played by the critical exponent $p^*=\frac{N-1}{k-1}$ in the case $1<k<N-1$. Precisely, if $0<p<p^*$ then no positive solutions of \eqref{half1} exist while if $p>p^*$, then \eqref{half1} admits always a solution. Such a solution is constructed by writing \eqref{l2} as
$$
u=\intl_{\R^N_+} G(\cdot ,y)d\mu(y)+\frac{2\lambda}{(N-2)\sigma_N} J_1\circ H_u \quad \mbox{ a.e. in }\R^N_+,
$$
where $H_u$ is defined in \eqref{hu} and $J_1$ is the $1$-lifting operator from $L^q(\R^{N-1})$ into $C^{0,\gamma}(\R^N)$ which we introduce and study in Section 2.2 below. Note that $H_u$ is a Riesz potential in $\R^{N-1}$ so that, combining the properties of $J_1$ and $H_u$ with a fixed point argument we obtain the existence of a positive solution to \eqref{half1}.

We next discuss the case $\mu\equiv 0$. The solutions of \eqref{half1} are required to satisfy \eqref{l1}, \eqref{l10} and \eqref{l2}. Notice that \eqref{l2} now reads
\begin{equation}\label{l3}
u(x)=\frac{2\lambda}{(N-2)\sigma_N} \intl_{\R^{N-1}} \intl_{\R^{N-1}} 
\frac{u(z',0)^p}{|x-(y',0)|^{N-2}|y'-z'|^k}dz' dy'\quad\mbox{ a.e. in }\R^N_+.
\end{equation}
Following the terminology introduced in \cite{CLO05, CLO06}, a solution $u$ of \eqref{l3} is called {\it regular} if  
$$
1<k<N-1 \quad\mbox{ and }\quad u(\cdot, 0)\in L^{\frac{2(N-1)}{k-1}}_{loc}(\R^{N-1}).
$$ 
Our main result in this case is stated below.

\begin{theorem}\label{thm2} {\sf (Case $\mu\equiv 0$)}
\smallskip

Assume $\mu\equiv 0$, $k\in (0, N-1)$ and $p,\lambda>0$.
\begin{enumerate}
\item[\rm (i)] The conclusion of Theorem \ref{thm1}(i)-(ii) holds.
\item[\rm (ii)] If 
\begin{equation}\label{iff}
1<k<N-1\quad\mbox{ and }\quad p> \frac{N-1}{k-1},
\end{equation}
then, there exists $C=C(N,\lambda,p,k)>0$ such that
\begin{equation}\label{munu0}
u(x)=C\intl_{\R^{N-1}} \frac{dy'}{|x-(y',0)|^{N-2} |y'|^{1+\frac{N-k}{p-1}} } \quad\mbox{ a.e. in } \R^N_+,
\end{equation}
is a solution of \eqref{half1}. Moreover, $u$ given by \eqref{munu0} is regular if and only if $p>2\cdot \frac{N-1}{k-1}-1$.

\item[\rm (iii)] If $1<k<N-1$ and $0<p< 2\cdot \frac{N-1}{k-1}-1$, then \eqref{half1} has no regular solutions.
\item[\rm (iv)] If $1<k<N-1$ and $p=2\cdot \frac{N-1}{k-1}-1$, then all regular solutions of \eqref{half1} are of the form
\begin{equation}\label{regs}
u_{\zeta',t}(x)=C(N,k,\lambda) \intl_{\R^{N-1}} \intl_{\R^{N-1}} 
\frac{\ds \Big(\frac{t}{t^2+|z'-\zeta'|^2}\Big)^{N-\frac{k+1}{2}}}{|x-(y',0)|^{N-2}|y'-z'|^k}dz' dy',
\end{equation}
where $\zeta'\in \R^{N-1}$ and $t>0$.
\end{enumerate} 
\end{theorem}
In relation to regular solutions, another critical exponent stands out from Theorem \ref{thm2}, namely 
$$
p^{**}=2\cdot \frac{N-1}{k-1}-1.
$$
We see from Theorem \ref{thm2} that 
the problem \eqref{half1} with $\mu\equiv 0$ has regular solutions if and only if $p\geq p^{**}$.

The remaining part of the manuscript is organised as follows. In Section 2 we collect some basic facts about the Riesz potentials, introduce the $\alpha$-lifting operator and derive some useful integral estimates. The proofs of Theorem \ref{thm1} and Theorem \ref{thm2} are provided in Section 3 and Section 4 respectively. 

Before we conclude this section, let us precise several notations that will be used throughout this paper.

\begin{itemize}
\item $B_R$ (resp. $B_R'$) denotes the open ball in $\R^N$ (resp. in $\R^{N_-1}$) centred at the origin and having radius $R>0$.

\item If $\nu$ is a Radon measure on $\R^N$, its Newtonian potential is denoted by 
$$
U^\nu(x)=\frac{1}{(N-2)\sigma_N}\intl_{\R^N}\frac{d\nu(y)}{|x-y|^{N-2}}.
$$
We also denote by $\overline \nu$ the reflection of $\nu$ with respect to the hyperplane $x_N=0$, that is, 
$\overline \nu(E)=\nu(F)$, where $F=\{x=(x', x_N)\in \R^N: \overline x=(x', -x_N)\in E\}$.

\item By $c,C, C_1, C_2$ we denote positive constants whose value may change on any occurrence. Whenever relevant, their dependence on the parameters will be given.
\end{itemize}

\section{Preliminary results}

\subsection{Some properties of Riesz potentials}
The Riesz potential of order $\alpha\in (0, N-1)$ of a measurable function $f:\R^{N-1}\to \R$  is defined by
\begin{equation}\label{rff}
I^{N-1}_\alpha f(x')=\intl_{\R^{N-1}}\frac{f(y')}{|x'-y'|^{N-1-\alpha}} dy'\quad \mbox{  for all } x'\in \R^{N-1}.
\end{equation}
We notice (see \cite[Lemma 2.1]{GLMM23}) that $I^{N-1}_\alpha f$ is finite as long as $f$ satisfies the integral condition
\begin{equation}\label{s0}
\intl_{\R^{N-1}}\frac{f(y')}{1+|y'|^{N-1-\alpha}} dy'<\infty.
\end{equation}
In particular, \eqref{s0} is fulfilled for all $f\in L^1(\R^{N-1})$. Furthermore, the following Sobolev inequality holds (see \cite[Theorem 1, Chapter 5]{S71}):
\begin{equation}\label{HLS1}
\|I_\alpha^{N-1} f\|_{L^{q}(\R^{N-1})}\leq C(N,s,\alpha) \|f\|_{L^s(\R^{N-1})}\quad\mbox{ for all }f\in L^s(\R^{N-1}),
\end{equation}
where
\begin{equation}\label{HLS2}
1<s<\frac{N-1}{\alpha}\quad\mbox{ and}\quad q=\frac{(N-1)s}{N-1-\alpha s}.
\end{equation}

 \subsection{The lifting operator}

Let $f:\R^{N-1}\to \R$ and $\alpha\in (0, N-1)$. We define the $\alpha$-lifting of $f$ in $\R^N$ by 
\begin{equation}\label{lift1}
J_\alpha f(x)=\intl_{\R^{N-1}}\frac{f(y')}{|x-(y',0)|^{N-1-\alpha}} dy'\quad\mbox{ for all }x\in \R^N.
\end{equation}
In this way, $f$ is mapped  to a function $J_\alpha f$ of $N$ variables. Furthermore, if $f$ satisfies \eqref{s0}  then $J_\alpha f$ is finite. Let us also note that 
\begin{equation}\label{liftr}
J_\alpha f(x',0)=I_\alpha^{N-1} f(x') \quad\mbox{ for all }x'\in \R^{N-1},
\end{equation}
which explains the $\alpha$-lifting appellation of our mapping $J_\alpha$. The main property of the $\alpha$-lifting operator $J_\alpha$ is stated in the proposition below.
\begin{proposition}
\label{lift2}
Let $\alpha\in (0, N-1)$ and $q>1$ be such that 
\begin{equation}\label{qq}
\alpha-1<\frac{N-1}{q}<\alpha.
\end{equation} 
If $f\in L^q(\R^{N-1})$ satisfies $J_\alpha f<\infty$ in $\R^N$,  then
\begin{equation}\label{lift0}
|J_\alpha f(x)-J_\alpha f(z)|\leq C|x-z|^{\alpha-\frac{N-1}{q}}\|f\|_{L^q(\R^{N-1})}\quad\mbox{ for all }x,z\in \R^N.
\end{equation}
In particular,  $J_\alpha f\in C^{0, \alpha-\frac{N-1}{q} }(\R^N)$.
\end{proposition}
\begin{proof}
Let $x,z\in \R^N$, $x\neq z$ and denote $R=|x-z|>0$. Then, for all $\zeta\in \R^N$ we have
$$
J_\alpha f(\zeta)=V_1(\zeta)+V_2(\zeta),
$$
where
$$
\begin{aligned}
V_1(\zeta)&=\intl_{|x-(y',0)|<2R}\frac{f(y') }{|\zeta-(y',0)|^{N-1-\alpha}} dy',\\[0.3cm]
V_2(\zeta)&=\intl_{|x-(y',0)|\geq 2R}\frac{f(y')}{|\zeta-(y',0)|^{N-1-\alpha}} dy'.
\end{aligned}
$$
Then,
\begin{equation}\label{lift3}
|J_\alpha f(x)-J_\alpha f(z)|\leq |V_1(x)|+|V_1(z)|+ |V_2(x)-V_2(z)|.
\end{equation}
By H\"older's inequality we estimate
\begin{equation}\label{lift4}
\begin{aligned}
|V_1(x)|&\leq \Big( \intl_{|x-(y',0)|<2R} |x-(y',0)|^{-q'(N-1-\alpha)} dy'\Big)^{1/q'}\|f\|_{L^{q}(\R^{N-1})}\\[0.3cm]
&\leq \Big( \intl_{|x'-y'|<2R} |x'-y'|^{-q'(N-1-\alpha)}dy'\Big)^{1/q'}\|f\|_{L^{q}(\R^{N-1})}\\[0.3cm]
&=C \Big( \intl_0^{2R} t^{N-2-q'(N-1-\alpha)} dt\Big)^{1/q'}\|f\|_{L^{q}(\R^{N-1})}\\[0.2cm]
&\leq C R^{\frac{N-1}{q'}-(N-1-\alpha)} \|f\|_{L^{q}(\R^{N-1})}\\
&\leq C R^{\alpha-\frac{N-1}{q} } \|f\|_{L^{q}(\R^{N-1})}.
\end{aligned}
\end{equation}
Further, we have
$$
|V_1(z)|\leq \Big( \intl_{|x-(y',0)|<2R} |z-(y',0)|^{-q'(N-1-\alpha)}dy' \Big)^{1/q'}\|f\|_{L^{q}(\R^{N-1})}.
$$
We next observe that
$$
|x-z|=R\mbox{ and } |x-(y',0)|<2R\Longrightarrow |z-(y',0)|\leq |x-z|+|x-(y',0)|<3R
$$
and  as before we find
\begin{equation}\label{lift5}
\begin{aligned}
|V_1(z)|&\leq \Big( \intl_{|z-(y',0)|<3R} |z-(y',0)|^{-q'(N-1-\alpha)} dy'\Big)^{1/q'}\|f\|_{L^{q}(\R^{N-1})}\\[0.3cm]
&\leq \Big( \intl_{|z'-y'|<3R} |z'-y'|^{-q'(N-1-\alpha)} dy' \Big)^{1/q'}\|f\|_{L^{q}(\R^{N-1})}\\[0.3cm]
&\leq C R^{\alpha-\frac{N-1}{q}} \|f\|_{L^{q}(\R^{N-1})}.
\end{aligned}
\end{equation}
To estimate the last term in \eqref{lift3} we use the Mean Value Theorem for $h(\zeta)=|\zeta-(y',0)|^{\alpha-N+1}$, $\zeta\in \R^N$, to deduce
$$
\begin{aligned}
\Big||x-(y',0)|^{\alpha-N+1}-|z-(y',0)|^{\alpha-N+1}\Big|&=|h(x)-h(z)|\\[0.2cm]
&=|x-z||\nabla h(tx+(1-t)z)|\\[0.2cm]
&= (N-1-\alpha)R|tx+(1-t)z-(y',0)|^{\alpha-N},
\end{aligned}
$$
for some $0<t<1$. Thus, if $|x-z|=R$ and $|x-(y',0)|\geq 2R$ we find
$$
|tx+(1-t)z-(y',0)|\geq |x-(y',0)|-(1-t)|x-z|\geq \frac{1}{2}|x-(y',0)|.
$$
Combining the last two estimates we find
$$
\Big||x-(y',0)|^{\alpha-N+1}-|z-(y',0)|^{\alpha-N+1}\Big|\leq CR |x-(y',0)|^{\alpha-N},
$$
and by H\"older's inequality we derive
$$
\begin{aligned}
|V_2(x)-V_2(z)|&\leq CR \intl_{|x-(y',0)|\geq 2R} |x-(y',0)|^{\alpha-N}|f(y')|dy'\\[0.3cm]
&\leq CR \Big( \intl_{|x-(y',0)|\geq 2R} |(x-(y',0)|^{-q'(N-\alpha)} dy'\Big)^{1/q'}\|f\|_{L^{q}(\R^{N-1})}.
\end{aligned}
$$
Since
$$
\begin{aligned}
|x-(y',0)|&=\big(|x'-y'|^2+x_N^2\big)^{1/2}\geq \frac{1}{\sqrt{2}}\big(|x'-y'|+|x_N|\big),
\end{aligned}
$$
we obtain from the above estimate that
\begin{equation}\label{lift6}
|V_2(x)-V_2(z)|\leq CR \Big( \intl_{|x-(y',0)|\geq 2R} \big(|x'-y'|+|x_N|\big)^{-q'(N-\alpha)}dy'\Big)^{1/q'}\|f\|_{L^{q}(\R^{N-1})}.
\end{equation}

\noindent{\bf Case 1:} $|x_N|\geq 2R$. Then, \eqref{lift6} yields 
$$
\begin{aligned}
|V_2(x)-V_2(z)|&\leq CR \Big( \intl_{\R^{N-1}} \big(|x'-y'|+|x_N|\big)^{-q'(N-\alpha)} dy'\Big)^{1/q'}\|f\|_{L^{q}(\R^{N-1})}\\[0.3cm]
&\leq CR \Big( \intl_0^\infty  \big(t+|x_N|\big)^{N-2-q'(N-\alpha)} dt\Big)^{1/q'}\|f\|_{L^{q}(\R^{N-1})}\\[0.3cm]
&\leq CR  |x_N|^{\frac{N-1}{q'}-(N-\alpha)}\|f\|_{L^{q}(\R^{N-1})}\\[0.3cm]
&\leq CR^{\alpha-\frac{N-1}{q}}\|f\|_{L^{q}(\R^{N-1})}.
\end{aligned}
$$
\noindent{\bf Case 2:} $|x_N|< 2R$. Note that 
$$
|x-(y',0)|\geq 2R\Longrightarrow |x'-y'|+|x_N|\geq |x-(y',0)|\geq 2R \Longrightarrow |x'-y'|\geq 2R-|x_N|.
$$
Thus, from \eqref{lift6} we obtain
$$
\begin{aligned}
|V_2(x)-V_2(z)|&\leq CR \Big( \intl_{|x'-y'|\geq 2R-|x_N|} \big(|x'-y'|+|x_N|\big)^{-q'(N-\alpha)} dy'\Big)^{1/q'}\|f\|_{L^{q}(\R^{N-1})}\\[0.3cm]
&\leq CR \Big( \intl_{2R-|x_N|}^\infty  \big(t+|x_N|\big)^{N-2-q'(N-\alpha)} dy' \Big)^{1/q'}\|f\|_{L^{q}(\R^{N-1})}\\[0.3cm]
&= CR  (2R)^{\frac{N-1}{q'}-(N-\alpha)}\|f\|_{L^{q}(\R^{N-1})}\\[0.3cm]
&= CR^{\alpha-\frac{N-1}{q} }\|f\|_{L^{q}(\R^{N-1})}.
\end{aligned}
$$
Hence, in both the above cases we have 
$$
|V_2(x)-V_2(z)|\leq CR^{\alpha-\frac{N-1}{q}}\|f\|_{L^{q}(\R^{N-1})}.
$$
Finally, we use the above estimate together with \eqref{lift5} and \eqref{lift4} into \eqref{lift3} to reach the conclusion.
\end{proof}

Using \eqref{liftr} we derive from Proposition \ref{lift2} that 
\begin{remark}\label{remm}
Suppose $q>1$ satisfies \eqref{qq} and let $g\in L^q(\R^{N-1})$ be such that $J_\alpha g<\infty$ in $\R^N$. Then 
$I_\alpha^{N-1}g\in C(\R^N)$.
\end{remark}
Combining the estimates \eqref{HLS1} and \eqref{lift0} we also have:

\begin{corollary}\label{corlift}
Let $\alpha\in (0, N-1)$ and $s>1$ be such that
$$
\frac{N-1}{\alpha+1}<s<\frac{N-1}{\alpha}.
$$
Let $q>1$ be given by \eqref{HLS2} and $g\in L^s(\R^{N-1})$ be such that $J_1 (I_\alpha^{N-1}g)<\infty$ in $\R^N$. 
Then, there exists a constant $C(N,s,\alpha)>0$ such that 
$$
\sup_{\substack{x,z\in \R^N\\ x\neq z}}\frac{|J_1 (I_\alpha^{N-1}g)(x)-J_1(I_\alpha^{N-1}g)(z)|}{|x-z|^{1-\frac{N-1}{q}}}\leq C(N,s,\alpha)\|g\|_{L^s(\R^{N-1})}.
$$
\end{corollary}
\begin{proof} Let $f=I_\alpha^{N-1}g$.  Then, by our assumption on $s$ we have $0<\frac{N-1}{q}<1$. By \eqref{lift0}  and by \eqref{HLS1}, we find
$$
\sup_{\substack{x,z\in \R^N\\ x\neq z}}\frac{|J_1 f(x)-J_1f(z)|}{|x-z|^{1-\frac{N-1}{q}}}\leq c\|f\|_{L^q(\R^{N-1})}= c\|I_\alpha^{N-1}g\|_{L^q(\R^{N-1})}
\leq C\|g\|_{L^s(\R^{N-1})},
$$
where $c,C>0$ depend only on $N$,  $s$ and $\alpha$.
\end{proof}

\subsection{Some useful integral estimates}
In this section we derive two integral estimates, the first of which extends the estimates in \cite{G22, G24, GKS20}.
\begin{lemma}\label{lint1}
Let $1<k<N-1<\beta$. Then, there exists $C_1=C_1(N,k, \beta)>0$ such that
\begin{equation}\label{stan1}
\intl_{\R^{N-1}}\frac{dz'}{|y'-z'|^{k} (1+|z'|)^\beta }\leq C_1(1+|y'|)^{-k}\quad\mbox{ for all }y'\in \R^{N-1}.
\end{equation}

\end{lemma}
\begin{proof} Assume first $y'\in \R^{N-1}$, $|y'|>1$.
We decompose 

\begin{equation}\label{stan2}
\intl_{\R^{N-1}}\frac{dz'}{ |y'-z'|^{k} (1+|z'|)^\beta }=D_1+D_2+D_3,
\end{equation}
where 
$$
\begin{aligned}
D_1&=\intl_{|z'|<\frac{|y'|}{2}}\frac{dz'}{ |y'-z'|^{k} (1+|z'|)^\beta },\\[0.3cm]
D_2&=\intl_{\frac{|y'|}{2}\leq |z'|< 2|y'|}\frac{dz'}{|y'-z'|^{k} (1+|z'|)^\beta },\\[0.3cm]
D_3&=\intl_{|z'|\geq 2|y'|}\frac{dz'}{|y'-z'|^{k} (1+|z'|)^\beta }.
\end{aligned}
$$
To estimate $D_1$ we observe that $|z'|<\frac{|y'|}{2}$ yields
$\ds 
|y'-z'|\geq |y'|-|z'|> \frac{|y'|}{2}
$
and thus
$$
\begin{aligned}
D_1&\leq 2^{k} |y'|^{-k} \intl_{|z'|<\frac{|y'|}{2}}\frac{dz'}{(1+|z'|)^\beta }=C|y'|^{-k}\intl_0^{\frac{|y'|}{2}}\frac{t^{N-2}}{(1+t)^\beta } dt\\[0.3cm]
&\leq C|y'|^{-k}\intl_0^{\frac{|y'|}{2}} (1+t)^{N-2-\beta } dt\leq C|y'|^{-k} \quad\mbox{ since }\beta>N-1.
\end{aligned}
$$
Next, using the elementary inequality
\begin{equation}\label{stan3}
a\geq \frac{1+a}{2}\quad\mbox{ for all }a\in \R,\,  a>1,
\end{equation}
we find 
\begin{equation}\label{stan4}
D_1\leq C(1+|y'|)^{-k}.
\end{equation}
To estimate $D_2$ we notice that
$$
\frac{|y'|}{2}\leq |z'|<2|y'|\Longrightarrow 
\begin{cases}
|y'-z'|\leq |y'|+|z'|< 3|y'|\\[0.2cm]
1+|z'|\geq \frac{1+|y'|}{2}
\end{cases},
$$
so
\begin{equation}\label{stan5}
\begin{aligned}
D_2&\leq C(1+|y'|)^{-\beta} \intl_{|y'-z'|<3|y'|} \frac{dz'}{|y'-z'|^{k}}\\[0.3cm]
&= C(1+|y'|)^{-\beta} \intl_0^{3|y'|} t^{N-2-k} dt \\[0.3cm]
&\leq C(1+|y'|)^{-\beta} |y'|^{N-1-k} \leq C(1+|y'|)^{N-1-\beta-k}\\[0.2cm]
&\leq C(1+|y'|)^{-k}.
\end{aligned}
\end{equation}
Finally, to estimate $D_3$ we have
$|z'|\geq 2|y'|\Longrightarrow |y'-z'|\geq |y'|$
so that
\begin{equation}\label{stan5a}
\begin{aligned}
D_3&\leq |y'|^{-k} \intl_{|z'|\geq 2|y'|}\frac{dz'}{(1+|z'|)^\beta}\leq C |y'|^{-k} \intl_{2|y'|}^\infty (1+t)^{N-2-\beta} dt\\[0.3cm]
&\leq C\Big(\frac{1+|y'|}{2}\Big)^{-k} (1+2|y'|)^{N-1-\beta}
\leq C (1+|y'|)^{-k}.
\end{aligned}
\end{equation}
We now use \eqref{stan4}-\eqref{stan5a} into \eqref{stan2} to deduce \eqref{stan1} for $|y'|>1$. 

Assume now $|y'|\leq 1$ and let $g(z')=(1+|z'|)^{-\beta}$. Since $\beta>N-1$ we have $g\in L^q(\R^{N-1})$ for all $q\geq 1$. Thus, for 
$$
\alpha=N-1-k\in (0, N-1) \quad \mbox{ and }\quad \alpha-1<\frac{N-1}{q}<\alpha,
$$ 
we have that $g$ satisfies \eqref{s0}. By Remark \ref{remm} it follows that $I_\alpha^{N-1}g$ is continuous and positive and thus bounded on the compact set $\overline B_1'\subset \R^{N-1}$. Hence, there exists a positive constant $c>0$ such that 
$$
\intl_{\R^{N-1}}\frac{dz'}{ |y'-z'|^{k} (1+|z'|)^\beta }=(I^{N-1}_\alpha g)(y')\leq c(1+|y'|)^{-k}\quad\mbox{ for all }y'\in \R^{N-1}, |y'|\leq 1.
$$
This concludes our proof.
\end{proof}

\begin{lemma}\label{lint2}
Let $1<k<N-1$. Then, there exists $C_2=C_2(N,k)>0$ such that
\begin{equation}\label{stan6}
\intl_{\R^{N-1}}\frac{dy'}{ |x-(y',0)|^{N-2} (1+|y'|)^{k} }\leq C_2(1+|x|)^{1-k}\quad\mbox{ for all }x\in \R^{N}.
\end{equation}
\end{lemma}
\begin{proof} Let $x=(x', x_N)\in \R^N$. Assume first $|x|>1$ and decompose
\begin{equation}\label{stan7}
\intl_{\R^{N-1}}\frac{dy'}{|x-(y',0)|^{N-2} (1+|y'|)^{k} }=E_1+E_2+E_3,
\end{equation}
where
$$
\begin{aligned}
E_1&=\intl_{|y'|<\frac{|x|}{2}}\frac{dy'}{ |x-(y',0)|^{N-2} (1+|y'|)^{k}},\\[0.3cm]
E_2&=\intl_{\frac{|x|}{2}\leq |y'|< 2|x|}\frac{dy'}{ |x-(y',0)|^{N-2} (1+|y'|)^{k} },\\[0.3cm]
E_3&=\intl_{|y'|\geq 2|x|} \frac{dy'}{ |x-(y',0)|^{N-2} (1+|y'|)^{k}}.
\end{aligned}
$$
Let us notice first that 
$$
|y'|<\frac{|x|}{2}\Longrightarrow |x-(y',0)|\geq |x|-|y'|>\frac{|x|}{2},
$$
so that 
\begin{equation}\label{stan8}
\begin{aligned}
E_1&\leq  \Big(\frac{|x|}{2}\Big)^{2-N}  \intl_{|y'|<\frac{|x|}{2}}\frac{dy'}{(1+|y'|)^{k}}
\leq C|x|^{2-N}\intl_0^{\frac{|x|}{2}} (1+t)^{N-2-k} dt\\[0.3cm]
&\leq C|x|^{2-N}(1+|x|)^{N-1-k}\\[0.3cm]
&\leq C\Big(\frac{1+|x|}{2}\Big)^{2-N} (1+|x|)^{N-1-k} \quad\mbox{ by \eqref{stan3} with $a=|x|>1$}\\
&=C(1+|x|)^{1-k}.
\end{aligned}
\end{equation}
Next, we use the estimate
$$
|y'|<2|x| \Longrightarrow |x-(y',0)|\leq |x|+|y'|< 3|x|,
$$
and thus $|x'-y'|\leq |x-(y',0)|<3|x|$. Hence,
\begin{equation}\label{stan9}
\begin{aligned}
E_2&\leq  \Big(1+\frac{|x|}{2}\Big)^{-k}\intl_{\frac{|x|}{2}
\leq |y'|<2|x|}\frac{dy'}{|x-(y',0)|^{N-2}}\\[0.3cm]
&\leq C(1+|x|)^{-k}\intl_{|x'-y'|<3|x|}
\frac{dy'}{|x'-y'|^{N-2}}\\[0.3cm]
&= C(1+|x|)^{-k}\intl_0^{3|x|} 1 dt= C|x|(1+|x|)^{-k}\leq C(1+|x|)^{1-k}. 
\end{aligned}
\end{equation}
Finally, to estimate $E_3$ we use 
$$
|y'|\geq 2|x|\Longrightarrow |x-(y',0)|\geq |y'|-|x|\geq \frac{|y'|}{2},
$$
so that
\begin{equation}\label{stan10}
\begin{aligned}
E_3&\leq  2^{N-2}\intl_{|y'|\geq 2|x|}\frac{dy'}{(1+|y'|)^k |y'|^{N-2}}\\[0.3cm]
&\leq C\intl_{2|x|}^\infty \frac{dt}{(1+t)^{k}} \leq C(1+|x|)^{1-k}. 
\end{aligned}
\end{equation}
Using \eqref{stan8}-\eqref{stan10} into \eqref{stan7}, we deduce \eqref{stan6} for all $x\in \R^N$, $|x|>1$. 

Assume now $|x|\leq 1$. Then $|x'|\leq 1$ and 
\begin{equation}\label{stan11}
\begin{aligned}
\intl_{\R^{N-1}}\frac{dy'}{|x-(y',0)|^{N-2} (1+|y'|)^{k} }& \leq 
\intl_{\R^{N-1}}\frac{dy'}{|x'-y'|^{N-2} (1+|y'|)^{k} }\\[0.3cm]
&=(I_{1}^{N-1}g)(x'),
\end{aligned}
\end{equation}
where $I_1^{N-1}$ is the Riesz potential of order $\alpha=1$ in $\R^{N-1}$ as defined in \eqref{rff} and $g(x')=(1+|x'|)^{-k}$. Notice that $g$ satisfies \eqref{s0} with $\alpha=1$, so $J_1 g<\infty$ in $\R^N$. Since 
$$
g\in L^q(\R^{N-1}) \quad\mbox{ for all }\quad  q>N-1,
$$ 
by Remark \ref{remm} we deduce that $I_{1}^{N-1}g$ is continuous and positive in $\R^{N-1}$. Thus, $I_{1}^{N-1}g$ and $(1+|x'|)^{1-k}$ are both continuous and positive on the closed unit ball $\overline{B}_1'\subset \R^{N-1}$ and hence there exists $C>0$ such that
$$
\begin{aligned}
I_{1}^{N-1}g(x')
& \leq C(1+|x'|)^{1-k}\leq C\Big(\frac{1+|x'|+|x_N|}{2}\Big)^{1-k}\\[0.2cm]
&\leq C(1+|x|)^{1-k}  \quad \mbox{ for all } x\in \R^N, |x|\leq 1. 
\end{aligned}
$$
Combining this last estimate with \eqref{stan11} we derive \eqref{stan6} for $|x|\leq 1$ and conclude the proof.
\end{proof}

\section{Proof of Theorem \ref{thm1}}

Let $x'\in \R^{N-1}$ for which \eqref{l1} holds. Then, by Fatou Lemma, \eqref{l20}  and the fact that $\mu\not\equiv 0$ we find

\begin{equation}\label{mmb1}
\begin{aligned}
u(x', 0)&=\lim_{z\in \R^N_+, \, z\to (x',0)} u(z)\\[0.3cm]
&=\lim_{z\in \R^N_+, \, z\to (x',0)}  \Bigg\{\intl_{\R^N_+}G(z, y)d\mu(y) +\lambda \intl_{\R^{N-1}} G(z, (y',0)) H_u(y') dy'  \Bigg\}\\[0.3cm]
&\geq \intl_{\R^N_+} \liminf_{z\to (x', 0)} G(z, y)d\mu(y) +\lambda \intl_{\R^{N-1}} \liminf_{z\to (x', 0)} G(z, (y',0)) H_u(y') dy' \\[0.3cm]
&=\intl_{\R^N_+}G((x', 0), y)d\mu(y)+\lambda \intl_{\R^{N-1}}  G((x', 0), (y',0)) H_u(y') dy' \\[0.3cm]
&=\underbrace{\intl_{\R^N_+}G((x', 0), y)d\mu(y)}_{>0}+ \frac{2\lambda}{(N-2)\sigma_N}   \intl_{\R^{N-1}}  \frac{H_u(y')}{|x'-y'|^{N-2}} dy', \\[0.3cm]
\end{aligned}
\end{equation}
where $H_u$ is the integral operator defined in \eqref{hu}.
We see from \eqref{mmb1} that $u(x',0)>0$ a.e. in $\R^{N-1}$. 

Furthermore, \eqref{mmb1} and Fubini Theorem  yields
\begin{equation}\label{mmb2}
\begin{aligned}
u(x', 0) & \geq  \frac{2\lambda}{(N-2)\sigma_N}  \intl_{\R^{N-1}}  \frac{H_u(y')}{|x'-y'|^{N-2}} dy' \\[0.3cm]
& = C
\intl_{\R^{N-1}} u(z',0)^p \Bigg( \intl_{\R^{N-1}} 
\frac{dy'}{|x'-y'|^{N-2}|y'-z'|^k}\Bigg) dz' \quad\mbox{ a.e. in }\R^{N-1},
\end{aligned}
\end{equation}
where $C=\frac{2\lambda}{(N-2)\sigma_N} >0$. 
\subsection{Proof of Theorem \ref{thm1}(i1)}
Assume $0<k\leq 1$. Then,  for all $x', z'\in \R^{N-1}$ and $|y'|>\max\{|x'|, |z'| \}$ we have $|x'-y'|\leq |x'|+|y'|\leq 2|y'|$ and similarly $|y'-z'|\leq 2|y'|$. Thus, we have
$$
\begin{aligned}
\intl_{\R^{N-1}}  \frac{dy'}{|x'-y'|^{N-2}|y'-z'|^k} &\geq \intl_{|y'|>\max\{|x'|, |z'|  \}} 
\frac{dy'}{|x'-y'|^{N-2}|y'-z'|^k}\\[0.3cm]
&\geq 2^{2-N-k} \intl_{|y'|>\max\{|x'|, |z'| \}} 
\frac{dy'}{|y'|^{N-2+k}}\\[0.3cm]
&=C\intl_{\max\{|x'|, |z'| \}}^\infty t^{-k} dt=\infty.
\end{aligned}
$$
Thus, \eqref{mmb2} yields $u(x', 0)\equiv \infty$ in $\R^{N-1}$ and this contradicts \eqref{l1}. Hence, \eqref{half1} has no positive solutions if $0<k\leq 1$.

\subsection{Proof of Theorem \ref{thm1}(i2)}
Assume $1<k<N-1$ and $0<p<\frac{N-1}{k-1}$. We use the following formula (see \cite[Chapter 5.9]{LL10} or the Appendix in \cite{M23} for an elementary proof):
\begin{equation}\label{selb0}
\intl_{\R^{N-1}}  \frac{dy'}{|x'-y'|^{N-1-a}|y'-z'|^b} =\frac{C(N, a,b)}{|x'-z'|^{b-a}} \quad\mbox{ for all }x', z'\in \R^{N-1},
\end{equation}
for all $0<a<b<N-1$.
In particular, for $a=1$ and $b=k$ we have
\begin{equation}\label{selb}
\intl_{\R^{N-1}}  \frac{dy'}{|x'-y'|^{N-2}|y'-z'|^k} =\frac{C(N, k)}{|x'-z'|^{k-1}} \quad\mbox{ for all }x', z'\in \R^{N-1}.
\end{equation}
Thus, letting $v(x')=u(x',0)$, from the above equality and \eqref{mmb2} we have
\begin{equation}\label{seq0}
v(x')\geq c\intl_{\R^{N-1}}\frac{v(y')^p}{|x'-y'|^{k-1}} dy'\quad\mbox{ a.e.  }x'\in \R^{N-1},
\end{equation}
where $c>0$. Let $|x'|<1$ for which $v(x')<\infty$ (see \eqref{l1}). Then
$$
\infty>v(x')\geq c\intl_{|y'|<1}\frac{v(y')^p}{|x'-y'|^{k-1}} dy' \geq 2^{1-k} c\intl_{|y'|<1}v(y')^p dy',
$$
so that
\begin{equation}
\label{vs}
v^p\in L^1(B_1').
\end{equation}
\begin{lemma}\label{kk}
There exists $C>0$ such that
\begin{equation}\label{kk1}
v(x')\geq C |x'|^{1-k}\quad\mbox{ for a.e. }x'\in \R^{N-1}, |x'|> 1.
\end{equation}
\end{lemma}
\begin{proof} Let us notice first that from \eqref{mmb1} we have $v>0$ a.e. in $\R^{N-1}$. Hence, from \eqref{seq0} and \eqref{vs}, for all $|x'|>1$ we find
$$
\begin{aligned}
v(x')&\geq c\intl_{|y'|<1}\frac{v(y')^p}{|x'-y'|^{k-1}} dy'\geq c(2|x'|)^{1-k} \intl_{|y'|<1} v(y')^p  dy' \geq C|x'|^{1-k},
\end{aligned}
$$
where $C>0$ is a constant.
\end{proof}
Let $\{\gamma_n\}_{n\geq 0}$ be the sequence defined by $\gamma_0=k-1$ and
\begin{equation}\label{seqn1}
\gamma_{n+1}=p\gamma_n+k-N \quad\mbox{ for all }n\geq 0.
\end{equation}
\begin{lemma}\label{lseq}
We have:
\begin{enumerate}
\item[\rm (i)] The sequence $\{\gamma_n\}_{n\geq 0}$ is decreasing and 
$$\ds \lim_{n\to \infty}\gamma_n=\begin{cases}
\ds \frac{N-k}{p-1} &\quad\mbox{ if }0<p<1\\[0.2cm]
\ds -\infty &\quad\mbox{ if } 1\leq p<\frac{N-1}{k-1}
\end{cases}.
$$
\item[\rm (ii)] Whenever $n\geq 0$ is such that  $\gamma_n>0$, there exists $C_n>0$ so that 
\begin{equation}\label{seqn}
v(x')\geq C_n |x'|^{-\gamma_n} \quad\mbox{ for a.e. } x'\in \R^{N-1}, |x'|> 1. 
\end{equation}
\end{enumerate}
\end{lemma}
\begin{proof}
(i) Since $0<p<\frac{N-1}{k-1}$, it is easy to check that $\{\gamma_n\}_{n\geq 0}$ is decreasing and thus, there exists $L:=\lim_{n\to \infty}\gamma_n\in \R\cup\{-\infty\}$. If $0<p<1$ then by induction $\gamma_n\geq \frac{N-k}{p-1}$. Thus, $\{\gamma_n\}_{n\geq 0}$ is convergent and passing to the limit in \eqref{seqn1} we find $L=\frac{N-k}{p-1}$. If $p=1$, then sequence cannot be convergent, otherwise, passing to the limit in \eqref{seqn1} we reach a contradiction; thus $L=-\infty$ in this case. If $1<p<\frac{N-1}{k-1}$ we can check that $\gamma_n<\frac{N-k}{p-1}$ for all $n\geq 0$ and then $L=-\infty$.

(ii) We prove \eqref{seqn} by induction. For $n=0$, the estimate \eqref{seqn} follows from Lemma \ref{kk}. Assume  now \eqref{seqn} holds for some $n\geq 0$ with $\gamma_n>0$ and $\gamma_{n+1}>0$. Then, if $|x'|>1$, from \eqref{seq0} we estimate
$$
\begin{aligned}
v(x')&\geq C_n\intl_{|y'|>|x'|}\frac{v(y')^p}{|x'-y'|^{k-1}} dy'\geq C_n\intl_{|y'|>|x'|}\frac{|y'|^{-p\gamma_n}}{|x'-y'|^{k-1}} dy'\\[0.3cm]
&\geq 2^{1-k} C_n\intl_{|y'|>|x'|}|y'|^{-p\gamma_n-k+1}dy',
\end{aligned}
$$
and thus
\begin{equation}\label{kk2}
v(x')\geq C_n\intl_{|x'|}^\infty t^{N-p\gamma_n-k-1}dt=C_n\intl_{|x'|}^\infty t^{-\gamma_{n+1}-1}dt=C_n|x'|^{-\gamma_{n+1}},
\end{equation}
since $\gamma_{n+1}>0$. This completes our proof by induction.
\end{proof}

\noindent{\bf Proof of Theorem \ref{thm1}(i2) completed.}  Let $\{\gamma_n\}_{n\geq 0}$ be the sequence defined in \eqref{seqn1}. Since by Lemma \ref{lseq}(i) we have $\lim_{n\to \infty}\gamma_n<0$ we can find $n\geq 0$ such that $\gamma_n>0\geq \gamma_{n+1}$. Then, \eqref{seqn} holds for $\gamma_n$. We proceed as in the proof of Lemma \ref{lseq}(ii) and from the estimate \eqref{kk2} we find
$$
v(x')\geq C\intl_{|x'|}^\infty t^{-\gamma_{n+1}-1}dt=\infty \quad\mbox{ since } \gamma_{n+1}\leq 0.
$$
Thus, \eqref{half1} has no solutions. \qed

\subsection{Proof of Theorem \ref{thm1}(ii)}
Assume $1<k<N-1$ and $p=\frac{N-1}{k-1}$. We notice that the estimate \eqref{kk1} in Lemma \ref{kk} still holds. Then, for $p=\frac{N-1}{k-1}$ we find 
$$
\intl_{\R^{N-1}} v(x')^p dx'\geq C\intl_{|x'|>1} |x'|^{1-N} dx'=\infty,
$$
which contradicts $u(\cdot, 0)=v\in L^p(\R^{N-1})$.\qed

\subsection{Proof of Theorem \ref{thm1}(iii)}
Throughout this section we assume
\begin{equation}\label{pk}
1<k<N-1\quad\mbox{ and } \quad p>\frac{N-1}{k-1}.
\end{equation}
We fix $s>1$ such that
\begin{equation}\label{pk1}
\frac{N-1}{N-k}<s<\frac{N-1}{N-k-1},
\end{equation}
and choose 
\begin{equation}\label{pk2}
q=\frac{(N-1)s}{N-1-(N-k-1)s} \quad\mbox{ and }\quad \gamma=1-\frac{N-1}{q}.
\end{equation}
From the above we have 
\begin{equation}\label{pk3}
p(k-1)s>N-1\,, \quad q>N-1\quad \mbox{ and }\quad 0<\gamma<1.
\end{equation}
The proof of Theorem \ref{thm1}(iii) relies on the following result. 
\begin{lemma}\label{mn}
Let $\nu$ be a positive Radon measure on $\R^N$ such that its Newtonian potential $U^\nu$ is continuous on $\R^N$ and there exists $A>0$ such that 
\begin{equation}\label{mn1}
U^\nu(x)\leq A(1+|x|)^{1-k}\quad\mbox{ for all }x\in \R^N.
\end{equation}
Then, there exists $M>0$ and $\lambda^*=\lambda^*(N,M, A,k,p)>0$ such that for all $0<\lambda<\lambda^*$ and all $R\geq 1$ the integral equation 
\begin{equation}\label{mn2}
u(x)=U^\nu(x)+\frac{2\lambda}{(N-2)\sigma_N} \intl_{B'_R} \intl_{B'_R} 
\frac{u(z',0)^p}{|x-(y',0)|^{N-2}|y'-z'|^k}dz' dy',
\end{equation}
has a solution $u\in C(\overline B_R)$. Furthermore, $u$ has the following properties:
\begin{enumerate}

\item[\rm (i)] $u(x)\leq M(1+|x|)^{1-k}$ for all $x\in \overline B_R$.
\item[\rm (ii)] $u-U^\nu\in C^{0, \gamma}(\overline B_R)$ and $\|u-U^\nu\|_{C^{0, \gamma}(\overline B_R)}\leq C(N,M,A,k,p)$.
\end{enumerate}
\end{lemma}
\begin{proof}
Fix $R\geq 1$ and set
$$
\Sigma_R=\{ u\in C(\overline B_R): U^\nu(x)\leq u(x) \leq M(1+|x|)^{1-k}\mbox{ for all } x\in\overline B_R\},
$$ where $M>0$ is a suitable constant that will be chosen later in this proof.  For all $u\in \Sigma_R$ and $x\in \overline B_R$ we define 
$$
Tu(x)=U^\nu(x)+\frac{2\lambda}{(N-2)\sigma_N} \intl_{B'_R} \intl_{B'_R} 
\frac{u(z',0)^p}{|x-(y',0)|^{N-2}|y'-z'|^k}dz' dy'.
$$
\medskip

We divide the proof of Lemma \ref{mn} into four steps.

\noindent{\bf Step 1:} $T:\Sigma_R\to C(\overline B_R)$ is well defined. 

Indeed, we write
$$
Tu=U^\nu+\frac{2\lambda}{(N-2)\sigma_N} F_Ru,
$$
where
$$
F_Ru(x)=\intl_{B'_R} \intl_{B'_R} 
\frac{u(z',0)^p}{|x-(y',0)|^{N-2}|y'-z'|^k}dz' dy' \quad\mbox{ for all }x\in \overline B_R.
$$
Let $u\in \Sigma_R$. By  Lemma \ref{lint1} with $\beta=p(k-1)>N-1$ we have
$$
\intl_{B'_R} \frac{u(z',0)^p}{|y'-z'|^k}dz' \leq M^p  \intl_{\R^{N-1}}\frac{dz'}{|y'-z'|^{k} (1+|z'|)^{p(k-1)}  }\leq C_1 M^p (1+|y'|)^{-k}.
$$
Now, by Lemma \ref{lint2} we obtain for all $x\in \overline B_R$ that
\begin{equation}\label{v1}
F_Ru(x) \leq C_1 M^p \intl_{\R^{N-1}}\frac{dy'}{ |x-(y',0)|^{N-2} (1+|y'|)^{k}  }\leq C_1 C_2 M^p (1+|x|)^{1-k}.
\end{equation}
We should point out that the above constants $C_1, C_2>0$ are independent of $R\geq 1$, they only depend on $N,k$ and $p$ as stated in Lemma \ref{lint1} and Lemma \ref{lint2}.
On the other hand, we see that 
$$
F_Ru=J_1 f\quad \mbox{ where }\quad f=(I_{N-k-1}^{N-1} g)\chi_{B_R'} \quad\mbox{ and }g=u^p\chi_{B_R'},
$$
where $I_{N-k-1}^{N-1} $ and $J_1$ are the Riesz and the lifting operators introduced in Section 2 and $\chi_{B_R'}$ denotes the indicator function of the ball $B_R'\subset \R^{N-1}$. Note that by \eqref{pk3}$_1$ we have $g\in L^s(\R^{N-1})$ and
$$
\|g\|_{L^s(\R^{N-1})}\leq M^p \Big(\intl_{\R^{N-1}} (1+|x'|)^{-p(k-1)s} dx'\Big)^{1/s}\leq C(N,k,p,s)M^p.
$$
In virtue of Corollary \ref{corlift} we now deduce
\begin{equation}\label{v2}
\sup_{\substack{x,z\in \R^N\\ x\neq z}}\frac{|F_Ru (x)-F_Ru(z)|}{|x-z|^{\gamma}}\leq C(N,k,p,s)\|g\|_{L^s(\R^{N-1})}\leq C(N,k,p,s)M^p,
\end{equation}
where $\gamma=\gamma(N,k,s)\in (0,1)$ is defined in \eqref{pk2}$_2$.
In particular, $F_R u$ is a continuous mapping and from \eqref{v2} and \eqref{v1} we have
\begin{equation}\label{v3}
\begin{aligned}
\|Tu- U^\nu\|_{C^{0,\gamma}(\overline B_R)}&=\frac{2\lambda}{(N-2)\sigma_N} \|F_R u\|_{C^{0,\gamma}(\overline B_R)}\\[0.3cm]
&\leq \frac{2\lambda}{(N-2)\sigma_N}  \big(C_1C_2+C(N,k,p,s)\big)M^p.
\end{aligned}
\end{equation}
Hence, we have proved that $Tu\in C(\overline B_R)$ and that $Tu- U^\nu$ is bounded in $C^{0,\gamma}(\overline B_R)$.

\medskip

\noindent{\bf Step 2:} There exists $M>0$ and $\lambda^*=\lambda^*(N,M, A,k,p)>0$ such {\color{red} that}
$$
T(\Sigma_R)\subset \Sigma_R\quad\mbox{ for all } 0<\lambda<\lambda^*\mbox{ and }R\geq 1.
$$
From \eqref{v1}  and our hypothesis \eqref{mn1} we have
\begin{equation}\label{v4}
Tu(x)\leq \Big(A+\frac{2\lambda}{(N-2)\sigma_N} C_1C_2M^p\Big)(1+|x|)^{1-k} \quad\mbox{ for all }x\in \overline B_R,
\end{equation}
where $C_1, C_2>0$ depend only on $N,k,p$ but not on $R\geq 1$. Fix now $M>2A$ and define 
$$
\lambda^*=\frac{(N-2)\sigma_N(M-A)}{2C_1C_2M^p}.
$$
Then, \eqref{v4} yields $Tu(x)\leq M(1+|x|)^{1-k}$ for all $0<\lambda<\lambda^*$ and all $x\in \overline B_R$, which concludes our claim in this step.

\medskip

\noindent{\bf Step 3:} $T(\Sigma_R)$ is compactly contained in $C(\overline B_R)$.

We observe from \eqref{v3} that for all $u\in \Sigma_R$ we have $Tu-U^\nu$ is bounded in $C^{0,\gamma}(\overline B_R)$.  Thus,
$$
\begin{aligned}
T(\Sigma_R)&=U^\nu+F_R(\Sigma_R)\\
&\subset U^\nu+\{w\in C^{0,\gamma}(\overline B_R): \|w\|_{C^{0,\gamma}(\overline B_R)}\leq C\}\hookrightarrow C(\overline B_R).
\end{aligned}
$$
This shows that $T(\Sigma_R)$ is compactly contained in $C(\overline B_R)$.

\medskip

\noindent{\bf Step 4:} Existence of a solution to \eqref{mn2}.

By Schauder Fixed Point Theorem,  for all $0<\lambda<\lambda^*$ and all $R\geq 1$ we may find $u\in \Sigma_R$ such that $Tu=u$. This means that $u$ satisfies \eqref{mn2}. From the definition of $\Sigma_R$ and \eqref{v3} we deduce that $u$ has the properties (i) and (ii) in the statement of Lemma \ref{mn}.
\end{proof}

By Theorem 1.9 in \cite{L72}, there exists a sequence of increasing measures $\{\mu_n\}_{n\geq 1}$ such that:
\begin{itemize}
\item $U^{\mu_n}$ is continuous in $\R^N$.
\item $\{U^{\mu_n}\}_{n\geq 1}$ converges pointwise to $U^\mu$ in  $\R^N$.
\end{itemize}

Let $\nu_n=\mu_n+\overline\mu_n$. We observe from \eqref{munu} and the properties of sequence $\{\mu_n\}_{n\geq 1}$ that:
\begin{itemize}
\item $U^{\nu_n}=U^{\mu_n}+U^{\overline\mu_n}$ is continuous in $\R^N$.
\item We have 
$$
U^{\nu_n}(x)\leq U^{\mu}(x)+U^{\overline\mu}(x)\leq 2C(\mu)(1+|x|)^{1-k}\quad\mbox{ a.e. in }\R^N.
$$
\end{itemize}
We may thus apply Lemma \ref{mn} to deduce the existence of  $\lambda^*>0$ and $u_n\in C(\overline B_n)$ such that
\begin{equation}\label{mn7}
u_n(x)=U^{\nu_n}(x)+\frac{2\lambda}{(N-2)\sigma_N} \intl_{B'_n} \intl_{B'_n} 
\frac{u_n(z',0)^p}{|x-(y',0)|^{N-2}|y'-z'|^k}dz' dy' \quad\mbox{ in }\overline B_n,
\end{equation}
for all $0<\lambda<\lambda^*$.
Furthermore, $u_n$ satisfies
\begin{equation}\label{mn8}
u_n(x)\leq M(1+|x|)^{1-k}\quad\mbox{ for all }x\in \overline B_n
\end{equation}
and 
\begin{equation}\label{mn9}
u_n-U^{\nu_n}\in C^{0, \gamma}(\overline B_n) \quad \mbox{ with }\quad 
\|u_n-U^{\nu_n}\|_{C^{0, \gamma}(\overline B_n)}\leq C=C(N,M,\mu,k,p).
\end{equation}
Let us point out that the constants $M,C>0$ in \eqref{mn8} and \eqref{mn9} are independent on $n$. Since
$$
d\overline \mu_n(y)=d\mu_n(\overline y),
$$
we see that 
$$
U^{\nu_n}(x)=U^{\mu_n}(x)+U^{\overline \mu_n}(x)=\int_{\R^N_+} G(x,y)d\mu_n(y).
$$
Thus, by the properties of $\{\mu_n\}_{n\geq 1}$ we have 
\begin{equation}\label{mn10}
U^{\nu_n}(x)\to \intl_{\R^N_+} G(x,y)d\mu(y)\quad\mbox{ as }n\to \infty, \quad\mbox{ for all }x\in \R^N_+.
\end{equation}
By \eqref{mn9}, the sequence $\{u_n-U^{\nu_n}\}_{n\ge 1}$ is bounded in $C^{0, \gamma}(\overline B_1)$. Since the embedding $C^{0, \gamma}(\overline B_1)\hookrightarrow C(\overline B_1)$ is compact, one can find a subsequence $\{u_n^1-U^{\nu_n^1}\}_{n\ge 1}\subset \{u_n-U^{\nu_n}\}_{n\ge 1}$ which converges in $C(\overline B_1)$. Consider now the sequence $\{u_n^1-U^{\nu_n^1}\}_{n\ge 1}$ which is bounded in $C^{0, \gamma}(\overline B_2)$. By the compactness of the embedding $C^{0, \gamma}(\overline B_2)\hookrightarrow C(\overline B_2)$ one may now find a subsequence $\{u_n^2-U^{\nu_n^2}\}_{n\ge 2}\subset \{u_n^1-U^{\nu_n^1}\}_{n\ge 1}$ which converges in $C(\overline B_2)$. Inductively, for all $m\geq 1$ one may find a subsequence 
$$
\{u_n^m-U^{\nu_n^m}\}_{n\ge m}\subset \{u_n^{m-1}-U^{\nu_n^{m-1}}\}_{n\ge m-1}
$$
that converges in $C(\overline B_m)$. Consider now the diagonal sequence $\{u_n^n-U^{\nu_n^n}\}_{n\ge 1}$ which converges locally  in $C(\R^N)$. Note that by \eqref{mn10},  $\{U^{\nu_n^n}\}_{n\geq 1}$ converges pointwise to the Green potential of $\mu$. Hence,
$u_n^n=(u_n^n-U^{\nu_n^n})+U^{\nu_n^n}$ is pointwise convergent to some function $u$. Also, $u_n^n$ satisfies \eqref{mn8}. Thus, we may use the Lebesgue Dominated Convergence Theorem and pass to the limit in \eqref{mn7} (in which $u_n$ and $U^{\mu_n}$ are replaced with $u_n^n$ and $U^{\nu_n^n}$ respectively). We obtain 
$$
u(x)=\intl_{\R^N_+} G(x,y)d\mu(y)+\frac{2\lambda}{(N-2)\sigma_N} \intl_{\R^{N-1}} \intl_{\R^{N-1}} 
\frac{u(z',0)^p}{|x-(y',0)|^{N-2}|y'-z'|^k}dz' dy' 
$$
almost everywhere in $\R^N_+$.
Finally, we notice that from \eqref{mn8}, $u$ satisfies $u(x)\leq M(1+|x|)^{1-k}$ a.e. in $\R^N_+$. Thus, by the properties of the lifting operator in Corollary \ref{corlift} we deduce  \eqref{vgre}. This concludes the proof of Theorem \ref{thm1}. \qed

\section{Proof of Theorem \ref{thm2}}
Before we proceed to the proof of Theorem \ref{thm2}, we recall the following result obtained in \cite{CLO05,CLO06} related to the regular solutions of the integral equation 
\begin{equation}\label{clo}
v(x')=c\intl_{\R^{N-1}}  \frac{v(y')^p}{|x'-y'|^{k-1}} dy' \quad\mbox{ a.e. in }\R^{N-1},
\end{equation}
where $1<k<N-1$, $p>1$ and $c>0$. Recall that positive solutions $v$ of \eqref{clo} are called regular if $v\in L^{\frac{2(N-1)}{k-1}}_{loc}(\R^{N-1})$.

\begin{lemma}\label{cloo} {\rm (see \cite{CLO05,CLO06})}
\begin{enumerate}
\item[\rm (i)] If $1<p<p^{**}$, then \eqref{clo} has no regular solutions.

\item[\rm (ii)] If  $p=p^{**}$, then all regular solutions of \eqref{clo} are of the form
\begin{equation}\label{regs0}
v_{\zeta', t}(x')=c \Big(\frac{t}{t^2+|x'-\zeta'|^2}\Big)^{\frac{k-1}{2}},
\end{equation}
for some $\zeta'\in \R^{N-1}$ and $c,t>0$.
\end{enumerate}
\end{lemma}
Part (i) in Lemma \ref{cloo} was obtained in \cite{CLO05} while part (ii) was derived in \cite{CLO06}.

\noindent{\bf Proof of Theorem \ref{thm2}.} (i)-(ii) Let us rewrite \eqref{l3} in the form
\begin{equation}
\label{rew}
u(x) = \frac{2\lambda}{(N-2)\sigma_N} \intl_{\R^{N-1}}  \frac{H_u(y')}{|x-(y',0)|^{N-2}} dy'\quad \mbox{ a.e. in }\R^N_+,
\end{equation}
where $H_u$ was introduced in \eqref{hu}.
Let $x'\in \R^{N-1}$ be such that \eqref{l1} holds. As in the proof of \eqref{mmb1} we find 
\begin{equation}\label{mmb3}
\infty> u(x', 0) \geq C \intl_{\R^{N-1}}  \frac{H_u(y')}{|x'-y'|^{N-2}} dy',
\end{equation}
for some constant $C>0$.
This implies that for almost all $x'\in \R^{N-1}$ the function
$$
\R^{N-1} \ni y'\longmapsto \frac{H_u(y')}{|x'-y'|^{N-2}} \quad \mbox{ is integrable on } \R^{N-1}.
$$
Since
$$
\frac{H_u(y')}{|x-(y',0)|^{N-2}}\leq \frac{H_u(y')}{|x'-y'|^{N-2}}\quad\mbox{for all }x\in \R^N_+, 
$$
we can use the Lebesgue Dominated Convergence Theorem to pass to the limit in \eqref{rew} with $x_N\to 0$. Thus,
$$
u(x',0)=\frac{2\lambda}{(N-2)\sigma_N} \intl_{\R^{N-1}} \intl_{\R^{N-1}} 
\frac{u(z',0)^p}{|x'-y'|^{N-2}|y'-z'|^k}dz' dy'\quad\mbox{ a.e. in }\R^{N-1}.
$$
By Fubini Theorem, we have
$$
u(x',0)=\frac{2\lambda}{(N-2)\sigma_N} \intl_{\R^{N-1}} u(z',0)^p \Bigg(\intl_{\R^{N-1}} 
\frac{ dy'}{|x'-y'|^{N-2}|y'-z'|^k}\Bigg) dz' \quad\mbox{ a.e. in }\R^{N-1}.
$$
If $0<k\leq 1$, we showed in Section 3.1 that the inner integral in the above equality is divergent. Hence, we must have $1<k<N-1$ and by 
\eqref{selb} we derive
$$
u(x',0)= c\intl_{\R^{N-1}}\frac{u(y',0)^p}{|x'-y'|^{k-1}} dy'\quad\mbox{ a.e. in } \R^{N-1}.
$$
Thus, $v(x')=u(x',0)$ satisfies \eqref{clo} which is a particular case of the integral inequality \eqref{seq0}. We proved in Section 3.2 and in Section 3.3  that \eqref{clo} has no solutions if $1<k<N-1$ and $0<p<\frac{N-1}{N-k}$. Conversely, assume now that \eqref{iff} holds and let us prove that the function $u$ defined in \eqref{munu0} is a solution of \eqref{half1} for some constant $C>0$. As above, by Lebesgue Dominated Convergence Theorem and formula \eqref{selb}, \eqref{munu0} yields
\begin{equation}\label{leb}
u(x',0)=C\intl_{\R^{N-1}} \frac{dy'}{|x'-y'|^{N-2} |y'|^{1+\frac{N-k}{p-1}} }=c|x'|^{-\frac{N-k}{p-1}} \quad\mbox{ a.e. in } \R^{N-1}.
\end{equation}
Now, by formula \eqref{selb0} with $a=N-k-1<b=\frac{p(N-k)}{p-1}<N-1$ we have
$$
\begin{aligned}
\frac{2\lambda}{(N-2)\sigma_N} &\intl_{\R^{N-1}} \intl_{\R^{N-1}} 
\frac{u(z',0)^p}{|x-(y',0)|^{N-2}|y'-z'|^k}dz' dy'\\[0.3cm]
&=C \intl_{\R^{N-1}} \intl_{\R^{N-1}} 
\frac{|z'|^{-\frac{p(N-k)}{p-1}} }{|x-(y',0)|^{N-2}|y'-z'|^k}dz' dy'\\[0.3cm]
&=C \intl_{\R^{N-1}} \Bigg(\intl_{\R^{N-1}} 
\frac{dz'}{|y'-z'|^k |z'|^{\frac{p(N-k)}{p-1}} } \Bigg)\frac{dy' }{|x-(y',0)|^{N-2}}\\[0.3cm]
&=C \intl_{\R^{N-1}} \frac{dy' }{|x-(y',0)|^{N-2} |y'|^{1+\frac{N-k}{p-1}} }\\[0.3cm]
&=u(x).
\end{aligned}
$$
Thus, $u$ satisfies \eqref{l3}. Using \eqref{leb} we see that
$$
\begin{aligned}
u\mbox{ is regular }& \Longleftrightarrow  u(\cdot, 0)\in L^{\frac{2(N-1)}{k-1}}_{loc}(\R^{N-1}) \\
&\Longleftrightarrow u(\cdot, 0)\in 
L^{\frac{2(N-1)}{k-1}}(B_1')\\
&\Longleftrightarrow p>p^{**}.
\end{aligned}
$$
(iii) If $u$ is a regular solution of \eqref{half1}, then $v(x')=u(x', 0)$ is a regular solution of  \eqref{clo}. We proved in Section 3.2 that \eqref{clo} has no solutions if $0<k\leq 1$ or $0<p\leq 1$. For $1<k<N-1$ and $1<p<p^{**}$, the nonexistence of regular solutions follows from Lemma \ref{cloo}(i).

(iv) Assume $1<k<N-1$, $p=p^{**}$ and let $u$ be a regular solution of \eqref{half1}. Then $v(x')=u(x',0)$ is a regular solution of \eqref{clo}. Thus, by Lemma \ref{cloo}(ii) we have 
\begin{equation}\label{regss}
u(x', 0)=v_{\zeta',t}(x')=c \Big(\frac{t}{t^2+|z'-\zeta'|^2}\Big)^{\frac{k-1}{2}},
\end{equation}
for some $\zeta'\in \R^{N-1}$ and $c,t>0$. Now, \eqref{l3} and \eqref{regss} yield \eqref{regs}. 

Conversely, assume now that $u$ is given by \eqref{regs} and let us show that it satisfies \eqref{l3} for a conveniently chosen constant $C(N,k,\lambda)>0$. Indeed, from \eqref{regs} and \eqref{selb} for all $x'\in \R^{N-1}$ we have
\begin{equation}\label{ga1}
\begin{aligned}
u_{\zeta',t}(x',0)&=C \intl_{\R^{N-1}} \Big(\frac{t}{t^2+|z'-\zeta'|^2}\Big)^{N-\frac{k+1}{2}}\Bigg(\intl_{\R^{N-1}} 
\frac{\ds dy'}{|x'-y'|^{N-2}|y'-z'|^k}\Bigg) dz'\\[0.3cm]
&=C  \intl_{\R^{N-1}} 
\frac{\Big(\ds \frac{t}{t^2+|z'-\zeta'|^2}\Big)^{N-\frac{k+1}{2}} }{|x'-z'|^{k-1}}dz'.
\end{aligned}
\end{equation}
By Lemma \ref{cloo}(ii), the function $v_{\zeta',t}(x')$ defined in \eqref{regs0} is a solution of \eqref{clo} for some $c>0$, so 
\begin{equation}\label{ga2}
\begin{aligned}
\intl_{\R^{N-1}} 
\frac{\Big(\ds \frac{t}{t^2+|z'-\zeta'|^2}\Big)^{N-\frac{k+1}{2}} }{|x'-z'|^{k-1}}dz'&= \intl_{\R^{N-1}} 
\frac{v_{\zeta',t}(z')^{p} }{|x'-y'|^{k-1}}dy'\\[0.3cm]
&=\frac{ v_{\zeta',t}(x')}{c}\\[0.3cm]
&=C \Big(\frac{t}{t^2+|x'-\zeta'|^2}\Big)^{\frac{k-1}{2}}.
\end{aligned}
\end{equation}
Combining \eqref{ga1} and \eqref{ga2}, we find
\begin{equation}\label{regsss0}
u_{\zeta',t}(x',0)=C\Big(\frac{t}{t^2+|x'-\zeta'|^2}\Big)^{\frac{k-1}{2}} \quad\mbox{  in }\R^{N-1}.
\end{equation}
Hence,
\begin{equation}\label{regsss}
C\Big(\frac{t}{t^2+|x'-\zeta'|^2}\Big)^{N-\frac{k+1}{2}}=u_{\zeta',t}(x',0)^p \quad\mbox{  in }\R^{N-1}.
\end{equation}
Using \eqref{regsss} into \eqref{regs}, we find that $u$ satisfies \eqref{l3}. By the Lebesgue Dominated Convergence Theorem, conditions \eqref{l1} and \eqref{l10} are now easy to check. From \eqref{regsss0} we also have that $u_{\zeta', t}$ is regular. This concludes the proof of our Theorem \ref{thm2}.\qed

\section*{Acknowledgement}

The author would like to thank the anonymous Referees for pointing out some improvements that led to the current version of the manuscript.


\begin{thebibliography}{99}  



\bibitem{CF21} N.F. Casta\~neda-Centuri\'on and L.C.F. Ferreira,  On singular elliptic boundary value problems via a harmonic analysis approach, {\it J. Differential Equations} {\bf 299} (2021), 402-428.


\bibitem{CL24} Q. Chen and Y. Lei, Regularity and Liouville theorem on an integral equation of Allen-Cahn type, {\it Discrete Continuous Dyn. Systems},  {\bf 45}  (2025),  37-55.

\bibitem{CL21} Q. Chen and Y. Lei, Asymptotic estimates for an integral equation in theory of phase transition, {\it Nonlinearity} {\bf 34} (2021), 3953-3968.

\bibitem{CLO05} W. Chen, C. Li and B. Ou, Qualitative properties of solutions for an integral equation, {\it Discrete Continuous Dyn. Systems} {\bf 12} (2005), 347-354.


\bibitem{CLO06} W. Chen, C. Li and B. Ou, Classification of solutions for an integral equation, {\it Commun. Pure Appl. Math.} {\bf 59} (2006), 330-343.


\bibitem{CLO07} W. Chen, C. Li and B. Ou, Classification of solutions for a system of integral equation, {\it Commun. Partial Differential Equations} {\bf 30} (2007), 59-65.

\bibitem{CSF96} M. Chipot, I. Shafrir and M. Fila, On the solutions to some elliptic equations with nonlinear Neumann boundary conditions, {\it Adv. Differential Equations} {\bf 1} (1996), 91-110.



\bibitem{FMM13} L.C.F. Ferreira, E.S. Medeiros and M. Montenegro, On the Laplace equation with a supercritical nonlinear Robin boundary condition in the half-space, {\it Calc. Var. Partial Differential Equations} {\bf 47} (2013), 667-682.

\bibitem{G22} M. Ghergu, {\it Partial Differential Inequalities with Nonlinear Convolution Terms},  Springer Briefs in Mathematics, 2022. 

\bibitem{G24} M. Ghergu, The stationary Gierer–Meinhardt system in the upper half-space: existence, nonexistence and asymptotics, {\it Math. Annalen} {\bf 90} (2024), 2931-2971. 

\bibitem{GKS20} M. Ghergu, P. Karageorgis and G. Singh, Positive solutions for quasilinear elliptic inequalities and systems with nonlocal terms, {\it J. Differential Equations} {\bf 268} (2020), 6033-6066.

\bibitem{GLMM23} M. Ghergu, Z. Liu, Y. Miyamoto and V. Moroz,  Nonlinear inequalities with double Riesz potentials,
{\it Potential Anal.} {\bf  59} (2023), 97-112.

\bibitem{H14} J. Harada, Positive solutions to the Laplace equation with nonlinear boundary conditions on the half
space, {\it Calc. Var. Partial Differential Equations} {\bf 50} (2014), 399-435. 

\bibitem{H94} B. Hu, Nonexistence of a positive solution of the Laplace equation with a nonlinear boundary condition, {\it 
Differential Integral Equations} {\bf 7} (1994), 30-313.

\bibitem{L72} N.S. Landkof, Foundations of Modern Potential Theory, Springer-Verlag Berlin-Heidelberg, 1972.

\bibitem{LL16} Y. Lei and C. Li, Sharp criteria of Liouville type for some nonlinear systems, {\it Discrete Continuous Dyn. Systems} {\bf 36} (2016), 3277-3315.

\bibitem{YYL04} Y.Y. Li, Remark on some conformally invariant integral equations: the method of moving spheres, {\it J. European Math. Society} {\bf 6} (2004), 153-180.

\bibitem{LZ03} Y.Y. Li and L. Zhang, Liouville type theorems and Harnack type inequalities for semilinear elliptic equations, {\it J. d’Analyse Math.}  {\bf 90} (2003), 27-87.

\bibitem{LZ95} Y. Y. Li and  M. Zhu, Uniqueness theorems through the method of moving spheres, {\it  Duke Math. J.} {\bf 80} (1995), 383-417.

\bibitem{LL20} X. Liu and Y. Lei, Existence of positive solutions for integral systems of the weighted Hardy-Littlewood-Sobolev type, {\it Discrete Continuous Dyn. Systems} {\bf 40} (2020), 467-489.


\bibitem{LZ99} Y. Lou and M. Zhu, Classification of nonnegative solutions to some elliptic problems,  {\it 
Differential Integral Equations} {\bf 12} (1999), 601-612.

\bibitem{O96} B. Ou, Positive harmonic functions on the upper half space satisfying a nonlinear boundary condition, {\it 
Differential Integral Equations} {\bf  9} (1996), 1157-1164.

\bibitem{LL10} E.H. Lieb and M. Loss, {\it Analysis}, American Math. Soc. 14, Second Edition, 2010.


\bibitem{M23} N. Miyake, Effect of decay rates of initial data on the sign of solutions
to Cauchy problems of polyharmonic heat equations, {\it Math. Ann.} {\bf 387}(2023), 265-289.

\bibitem{S71} E.M.  Stein, Singular Integrals and Differentiability Properties of Functions, Princeton University Press, 1971.
\bibitem{TWZ22} S. Tang, L. Wang and M. Zhu, Nonlinear elliptic equations on the upper half space,
{\it Commun. Contemp. Math.} {\bf 24} (2022), Paper No. 2050085.
\end{thebibliography}
\end{document}